\documentclass[11pt, a4paper]{amsart}

\usepackage{upgreek}
\usepackage{mathrsfs}
\usepackage{amsmath}
\usepackage{amsthm}
\usepackage{amssymb}
\usepackage{amsfonts}
\usepackage{xcolor}
\usepackage{enumitem}
\usepackage{CJK}
\usepackage{bbm}
\usepackage{graphicx}
\usepackage{mathtools}
\usepackage{tikz}
\usepackage{tikz-cd}
\usetikzlibrary{matrix,arrows,decorations.pathmorphing}

\begin{CJK}{UTF8}{}
\gdef\yama{\mbox{\textbf{山}}}

\gdef\ten{\mbox{\textbf{天}}}
\end{CJK}

\newcommand{\C}{\mathbb{C}}
\renewcommand{\P}{\mathbb{P}}

\newcommand{\F}{\mathbb{F}}
\newcommand{\Het}{H^1_{\text{\'{e}t}}}

\newcommand{\Q}{\mathbb{Q}}

\newcommand{\bmu}{\mathbf{\upmu}}
\newcommand{\Z}{\mathbb{Z}}
\newcommand{\less}{\mbox{$<$}}
\newcommand{\lesseq}{\mbox{$\leq$}}

\DeclareMathOperator{\Aut}{Aut}
\DeclareMathOperator{\Out}{Out}

\DeclareMathOperator{\End}{End}
\DeclareMathOperator{\Hom}{Hom}

\DeclareMathOperator{\ord}{ord}
\DeclareMathOperator{\Gal}{Gal}

\DeclareMathOperator{\Div}{Div}

\DeclareMathOperator{\Jac}{Jac}
\DeclareMathOperator{\Supp}{Supp}
\DeclareMathOperator{\Stab}{Stab}

\DeclareMathOperator{\rank}{rank}
\DeclareMathOperator{\divisor}{div}

\newtheorem{introtheorem}{Theorem}

\newtheorem{theorem}{Theorem}
\newtheorem{lemma}[theorem]{Lemma}
\newtheorem{proposition}[theorem]{Proposition}
\newtheorem{corollary}[theorem]{Corollary}

\theoremstyle{definition}

\numberwithin{equation}{section}
\numberwithin{theorem}{section}

\title[Cyclic covers and Ihara's Question]{Cyclic covers and Ihara's Question}

\author[C.~Rasmussen]{Christopher Rasmussen}
\address{Department of Mathematics and Computer Science \\ Wesleyan University \\ Middletown CT, 06459 \\ United States}
\email{crasmussen@wesleyan.edu}

\author[A.~Tamagawa]{Akio Tamagawa}
\address{Research Institute for Mathematical Sciences \\Kyoto University \\ Kyoto 606-802 \\ Japan}
\email{tamagawa@kurims.kyoto-u.ac.jp}

\begin{document}
\begin{CJK}{UTF8}{min}

\begin{abstract}
Let $\ell$ be a rational prime. Given a superelliptic curve $C/k$ of $\ell$-power degree, we describe the field generated by the $\ell$-power torsion of the Jacobian variety in terms of the branch set and reduction type of $C$ (and hence, in terms of data determined by a suitable affine model of $C$). If the Jacobian is good away from $\ell$ and the branch set is defined over a pro-$\ell$ extension of $k(\bmu_{\ell^\infty})$ unramified away from $\ell$, then the $\ell$-power torsion of the Jacobian is rational over the maximal such extension.

By decomposing the covering into a chain of successive cyclic $\ell$-coverings, the mod $\ell$ Galois representation attached to the Jacobian is decomposed into a block upper triangular form. The blocks on the diagonal of this form are further decomposed in terms of the Tate twists of certain subgroups $W_s$ of the quotients of the Jacobians of consecutive coverings.

The result is a natural extension of earlier work by Anderson and Ihara, who demonstrated that a stricter condition on the branch locus guarantees the $\ell$-power torsion of the Jacobian is rational over the fixed field of the kernel of the canonical pro-$\ell$ outer Galois representation attached to an open subset of $\P^1$.
\end{abstract}

\maketitle

\section*{Introduction}

Let $k$ denote a number field and $\ell$ a rational prime number. We fix, once and for all, an algebraic closure $\bar{k}$ of $k$. We let $G_k$ denote the Galois group of $\bar{k}/k$. Let $S_0$ be a $G_k$-stable finite subset of $\P^1(\bar{k})$ satisfying $\{0, 1, \infty \} \subseteq S_0$. We consider pairs $(X, f)$ where
\begin{itemize}[itemsep=0pt]
\item $X/\bar{k}$ is a complete smooth irreducible curve,
\item $f \colon X \to \P^1_{\bar{k}}$ is a $\bar{k}$-morphism branched only over $S_0$,
\item the Galois closure of $f$ has degree a power of $\ell$.
\end{itemize}
Because such pairs admit a natural action from the pro-$\ell$ fundamental group of $\P^1 - S_0$, it is natural to study their properties by way of the canonical pro-$\ell$ Galois representation
\[ \Phi_{k,\ell,S_0} \colon G_k \longrightarrow \Out \bigl( \pi_1^{\text{pro-}\ell}(\P^1_{\bar{k}} - S_0) \bigr). \]
This is the main focus of \cite{Anderson-Ihara:1988}, where Anderson and Ihara demonstrate several strong results about the arithmetic of both the individual curves appearing in such pairs, as well as the entire category $\mathfrak{X}$ of such pairs.\footnote{In fact, Anderson and Ihara consider the more general possibility of $S_0 \subset \P^1(\C)$, although we do not consider transcendental cases here.} In studying the field extension fixed by the kernel of $\Phi_{k,\ell, S_0}$, they introduce a combinatorial object, $\mathbb{S}$, consisting of finite sets $S \subset \P^1_{\bar{k}}$ obtained by pulling back $S_0$ through a finite combination of certain linear fractional transformations and $\ell$th power mappings.

For example, the following proposition follows directly from their work (see \cite[Cor.~3.8.1 and Remark pgs.~292-293]{Anderson-Ihara:1988}).
\begin{introtheorem}\label{thm:AI}
Take $S_0 = \{0,1,\infty\}$. Suppose $k$ is fixed by the kernel of $\Phi_{\Q, \ell, S_0}$, and let $C/k$ be a geometrically irreducible curve with affine model
\[{y^{\ell^n} = \lambda f(x)},\]
where $f$ is a monic polynomial in $k[x]$. Let $J$ be the Jacobian of $C$. Suppose there exist (possibly identical) sets $S, S' \in \mathbb{S}$ such that $\lambda \in S$ and the roots of $f(x)$ are contained in $S'$. Then the $\ell$-power torsion of $J$ is rational over the fixed field of the kernel of $\Phi_{\Q,\ell, S_0}$.
\end{introtheorem}

We introduce the kanji characters $\yama$ (`yama') and $\ten$ (`ten') to describe the situation more precisely. The characters translate in meaning as `mountain' and `heaven,' respectively. Let $\bmu$, $\bmu_N$, $\bmu_{\ell^\infty}$ denote, respectively, the group of all roots of unity in $\bar{k}$, the subgroup of $N$th roots of unity, and the subgroup $\cup_{n \geq 1} \bmu_{\ell^n}$ of $\ell$-power roots of unity. We let $\ten := \ten(k,\ell)$ denote the maximal pro-$\ell$ extension of $k(\bmu_{\ell^\infty})$ which is unramified away from $\ell$. In case $S_0 = \{0, 1, \infty \}$, we write $\yama := \yama(k,\ell)$ for the fixed field of the kernel of $\Phi_{k,\ell, S_0}$. It is well-known that $\yama \subseteq \ten$ for any choice of $k$ and $\ell$. Ihara has asked whether $\yama = \ten$ in case $k = \Q$; that is, \emph{does the mountain reach the heavens}? Ihara's question has been partially resolved. In the case where $\ell$ is an odd regular prime, Sharifi demonstrates in \cite{Sharifi:2002} that the equality $\yama(\Q,\ell) = \ten(\Q,\ell)$ is a consequence of the Deligne-Ihara Conjecture, which is now known to hold due to the work of Brown \cite{Brown:2012}. Still, Ihara's question remains open in general.

The hypotheses on $C/k$ in Proposition \ref{thm:AI} force $C$, hence $J$, to have good reduction outside the prime $\ell$. In the present article, we provide a generalization, by weakening the condition on both $\lambda$ and the branch set. Rather than requiring that they belong to objects of $\mathbb{S}$, we show that it suffices for them to be rational over the extension of $k$ generated by such subsets. However, we must explicitly add the hypothesis that the Jacobian possesses good reduction outside $\ell$.
\begin{introtheorem}\label{thm:main_theorem}
Suppose $C/k$ is a geometrically irreducible curve with affine model
\[ y^{\ell^n} = f(x), \qquad f \in k(x), \]
with $n \geq 1$. Let $J$ denote the Jacobian variety of $C$, and let $S \subseteq \P^1(\bar{k})$ be the branch locus of $C$. Suppose further that $C$ satisfies the following conditions:
\begin{enumerate}[label=(\alph*),itemsep=0pt]
\item the variety $J$ has good reduction away from $\ell$,
\item every element of $S$ is $\ten(k,\ell)$-rational.
\end{enumerate}
Then $k(J[\ell^\infty]) \subseteq \ten(k,\ell)$.
\end{introtheorem}
We note the following slightly weaker version of Theorem \ref{thm:main_theorem}, to show that a sufficient condition for the conclusion may be given entirely in terms of arithmetic data of the affine model of $C$.
\begin{introtheorem}
Suppose $C/k$ is a curve with affine model
\[ y^{\ell^n} = \lambda f(x), \qquad f(x) \in k[x], \]
where $f(x)$ is monic. Let $S$ be the set of roots of $f(x)$ whose multiplicity is not divisible by $\ell^n$. Suppose that $f$ splits completely over $\ten$ and
\begin{enumerate}[label=(\alph*), itemsep=0pt]
\item at least one point of $C$ is totally ramified,
\item $\lambda$ is an $\ell$-unit and every $s \in S$ is an $\ell$-integer,
\item for all distinct $s, s' \in S$, $s-s'$ is an $\ell$-unit.
\end{enumerate}
Then the Jacobian $J$ of $C$ satisfies $k(J[\ell^\infty]) \subseteq \ten$.
\end{introtheorem}
\begin{proof}
The existence of a totally ramified point in the covering guarantees that $C$ is geometrically irreducible. As $f$ splits over $\ten$, the branch locus of $C$ is rational over $\ten$. The conditions (b) and (c) guarantee that $C$, hence $J$, has good reduction away from $\ell$. Thus, Theorem \ref{thm:main_theorem} applies.
\end{proof}

Let $C/k$ be a superelliptic curve of degree $\ell^n$ satisfying the hypotheses of Theorem \ref{thm:main_theorem}, and let $J$ be the Jacobian of $C$. In \S1, we reduce to the case where $C$ has a convenient affine model. In order to more finely study the $G_k$-action on $J[\ell]$, we organize the branch points of $C$ by their ramification index. This is done in the beginning of \S2. Also in \S2, we decompose the covering into a composition of degree $\ell$ coverings and verify that the Jacobians $J_s$ ($0 \leq s \leq n$) of the intermediate curves may be viewed as nested subvarieties of $J$. Moreover, we demonstrate that the arithmetic of the $\ell$-power torsion of $J$ is closely related to the arithmetic of the $\ell$-power torsion of the successive quotients of the $J_s$. Already this demonstrates that the representation $\rho_{J,\ell}$ of Galois action on $J[\ell]$ has a block upper triangular shape, with the diagonal blocks determined by the quotients $A_s := J_s/J_{s-1}$ ($1 \leq s \leq n$).

In \S3, we demonstrate that the $\ell$-torsion of each quotient $A_s$ has a $G_k$-stable subspace $W_s$. We establish a relationship between the arithmetic of $S$ and that of $W_s$. In \S4--6, we demonstrate that the representations $\rho_{A_s,\ell}$ are themselves block diagonal, with each block a Tate twist of the representation attached to $W_s$. This is enough to prove Theorem \ref{thm:main_theorem}. In \S7, we study further the arithmetic of $S$ and $W_s$; these two sets usually, but not always, define the same extension of $k$. In \S8, we settle a question posed in \cite{Malmskog-Rasmussen:2016} regarding Picard curves over $\Q$. This also gives an example where Theorem \ref{thm:main_theorem} applies, but Theorem \ref{thm:AI} does not.

\subsection*{Notation}
For any $G_k$-set $T$, we let $k(T)$ denote the extension of $k$ generated by $T$ (that is, the field fixed by the kernel of the associated representation $G_k \to \Aut(T)$). For example, condition (b) of Theorem \ref{thm:main_theorem} is equivalent to $k(S) \subseteq \ten(k,\ell)$.

Once and for all, we fix an algebraic closure $\bar{k}$ of $k$, and let $G_k$ denote the Galois group of $\bar{k}/k$. For any curve $X/k$, we write $\Div X$ for the (free abelian) group of divisors generated by $X(\bar{k})$ and write $\Div^0 X$ for the subgroup of divisors of degree $0$. The divisor of a function $f \in k(X)$ is denoted $\divisor_X f$. We abuse notation and write $X$ in places where the geometric pullback $X \times_k \bar{k}$ is intended. For example, if $g \in \bar{k}(X)$, we denote its divisor $\divisor_X g$. This mild abuse of notation is also used for other $k$-varieties (e.g., the varieties $J_s$, $A_s$ defined in \S2). Likewise, for a $k$-morphism of varieties $\Pi \colon Y \to X$ we let $\Pi$ also denote the geometric pullback $\Pi \times_k \bar{k}$. 

\subsection*{Acknowledgments}
This work was supported by the Research Institute for Mathematical Sciences (RIMS), a Joint Usage/Research Center located in Kyoto University. Much of this work was completed while the first author was visiting RIMS, and he appreciates the Institute's support and hospitality. The second author was partly supported by JSPS {\sc KAKENHI} Grant Numbers JP22340006, JP15H03609 and JST CREST Grant Number JPMJCR15D4.

\section{Normalization of Affine Form}
Let $C/k$ be a curve. In this paper, we say $C$ is superelliptic of degree $N \geq 2$ if $C$ admits an affine model of the form
\begin{equation}\label{eqn:affine_form}
y^N = f(x), \quad f \in k(x).
\end{equation}
The morphism $\Pi \colon C \to \P^1$ corresponding to the $x$-coordinate of this affine model is a branched covering of $\P^1$ defined over $k$. After base change to a finite extension which contains $\bmu_N$, $\Pi$ is a Galois cyclic covering with Galois group $\Z/N\Z$. Note: some authors define a curve $C/k$ to be superelliptic if $C$ admits a branched covering which yields a Galois cyclic covering after appropriate base change. These two definitions are \emph{not} equivalent in general, as the existence of such a covering does not guarantee an affine model of the form \eqref{eqn:affine_form}. (This fact may be mildly surprising to the reader, since the two conditions \emph{are} equivalent in many cases of interest, such as hyperelliptic curves or Picard curves.)

Suppose $C/k$ is a superelliptic curve satisfying the hypotheses of Theorem \ref{thm:main_theorem}. We start by reducing to the case where $f(x)$ is a polynomial and $\infty$ is not a branch point. The contents of this lemma are quite well-known; we simply provide the proof for convenience.
\begin{lemma}\label{lemma:niceform}
Under the hypotheses of Theorem \ref{thm:main_theorem}, without loss of generality we may assume that $C$ has an affine model of the form $y^{\ell^n} = f(x)$ such that $f \in k[x]$, the set $S$ coincides with the set of roots of $f$ (in particular, $\infty \not\in S$), and the multiplicity $n_\xi$ of any root $\xi$ of $f(x)$ satisfies $0 < n_\xi < \ell^n$.
\end{lemma}
\begin{proof}
Select any point $\xi_0$ in $\P^1(k) - S$. If $\xi_0 = \infty$, we do nothing. If $\xi_0 \neq \infty$, we apply the coordinate change $x \mapsto (x - \xi_0)^{-1}$, and we may now be sure that $\infty$ is not a branch point of the covering. As $C$ is defined over $k$, there exist $\lambda \in k^\times$ and monic polynomials $f_i \in k[x]$ such that $f = \lambda f_1^{\vphantom{-1}} f_2^{-1}$. For any $\xi \in \P^1(\bar{k})$, define
\[ e_\xi(f) := \begin{cases} \ord_{(x - \xi)} f & \xi \neq \infty \\ \deg f_1 - \deg f_2 & \xi = \infty \end{cases}. \]
Let $S_f = \Supp \divisor_{\P^1} f - \{\infty \}$; this is precisely the set of $\xi \neq \infty$ for which $e_{\xi}(f) \neq 0$. The branch locus $S$ need not agree with $S_f$, but can be described as
\[ S = \{ \xi \in S_f : e_\xi(f) \not\equiv 0 \pmod{\ell^n} \}. \]
Let $k'/k$ be a finite extension over which $f_1$ and $f_2$ split completely. Then $S_f \subset k'$ and $f$ has the form
\[ f(x) = \lambda \prod_{\xi \in S_f} (x- \xi)^{e_\xi(f)}. \]
For any $e \in \Z$, let $\Xi_e := \{ \xi \in S_f : e_\xi(f) = e \}$. By definition, $\Xi_0 = \varnothing$. Moreover, as the polynomials $f_i$ are $k$-rational, the set $\Xi_e$ is $G_k$-stable for any $e \in \Z$. For any integer $e$, let
\[ g_e := \prod_{\xi \in \Xi_e} (x-\xi). \]
Note that for all but finitely many $e$, the set $\Xi_e$ is empty and $g_e = 1$. The Galois stability of $\Xi_e$ implies $g_e \in k[x]$. For each $e$, define
\[ \hat{e} := \left\lceil - \frac{e}{\ell^n} \right\rceil. \]
Then always $0  \leq e + \hat{e} \ell^n < \ell^n$. Applying the coordinate change
\[ y \mapsto y \cdot \prod_{e \in \Z} g_e^{-\hat{e}} \]
we obtain a new model for $C$ with affine equation
\[ y^{\ell^n} = \hat{f}(x), \qquad \qquad \hat{f} = f \cdot \prod_e g_e^{\hat{e} \cdot \ell^n} \in k[x]. \]
Consequently,
\[ \hat{f} = \lambda \prod_{\xi \in S} (x - \xi)^{n_\xi}, \qquad 0 < n_\xi < \ell^n, \]
which demonstrates $C$ has a model of the desired form.
\end{proof}
For the remainder, we now assume $f(x)$ has the form guaranteed by Lemma \ref{lemma:niceform}. For each $0 \leq s \leq n$, let $C_s/k$ be the projective curve given by the affine model
\[ C_s : (y_s)^{\ell^s} = f(x). \]
Thus $C_0 = \P^1$ and $C_n = C$. The covering $\Pi$ decomposes into a composition of $k$-morphisms, each of degree $\ell$:
\begin{equation}\label{eqn:tower}
\begin{tikzcd}[column sep=0.3in]
C = C_n \arrow{r}{\pi_n} & C_{n-1} \arrow{r}{\pi_{n-1}} & C_{n-2} \arrow{r} & \cdots \arrow{r} & C_1 \arrow{r}{\pi_1} & C_0 = \P^1
\end{tikzcd}
\end{equation}
This induces a corresponding tower of function fields
\[ k(x) = k(x, y_0) \subset k(x, y_1) \subset \cdots \subset k(x,y_{n-1}) \subset k(x,y_n) = k(x,y), \]
subject to the relations $y_0 = f$, $y_s^\ell = y_{s-1}$ for each $1 \leq s \leq n$, and $y_n = y$.

For each $1 \leq s \leq n$ we let $\Pi_s := \pi_1 \circ \cdots \circ \pi_s \colon C_s \to \P^1$. (So $\Pi = \Pi_n$.) By this arrangement, the branch locus of $\Pi_s$ is necessarily a subset of the branch locus of $\Pi_t$ whenever $1 \leq s \leq t \leq n$.

For $\xi \in S$, set $n_\xi := \ord_{(x - \xi)} f$, and set
\begin{equation}\label{eqn:n_infty}
n_\infty := \sum_{\xi \in S} n_\xi = \deg f.
\end{equation}
Thus,
\begin{equation}\label{eqn:f_divisor}
\divisor_{\P^1} f = \sum_{\xi \in S} n_\xi \cdot \xi - n_\infty \cdot \infty.
\end{equation}
As $\infty$ is not a branch point of $\Pi$, $n_\infty \equiv 0 \pmod{\ell^n}$.

\section{Stratification of the branch locus}\label{sec:S}

We now organize the branch locus of $\Pi$ by ramification index. For any $0 \leq j \leq n-1$, we define:
\[ \begin{split}
S[j] & :=  \{ \xi \in S : \ord_\ell n_\xi = j \} \\
r_j & :=  \# S[j]
\end{split}\]
Thus, $\sum_{j=0}^{n-1} r_j = \#S$. Further, we define for any $1 \leq s \leq n$ and any $0 \leq t \leq (n-1)$,
\[S[\less s] :=  \bigcup_{j=0}^{s-1} S[j], \qquad \qquad S[\lesseq t] := \bigcup_{j=0}^t S[j].  \]
By definition, any point chosen from $\Pi_n^{-1}(S[j])$ will have ramification index $\ell^{n-j}$. We observe that the branch set for $\Pi_s$ is exactly $S[\less s]$, and that the set $S[s]$ is $G_k$-stable, since $\Pi_n$ is defined over $k$.
\begin{lemma}\label{lemma:S0-nonempty}
The set $S[0]$ is non-empty. In fact, $\# S[0] \geq 2$.
\end{lemma}
\begin{proof}
Suppose $S[0]$ were empty. Then for every $\xi \in S$, $\ell \mid n_\xi$. Let $L/k$ be an extension containing a primitive $\ell$th root of unity $\zeta$ and an $\ell$th root of $\lambda$. Then there exists $h \in L[x]$ such that $h^\ell = f$. Moreover, the curve $C \times_k L$ must be reducible, as it is the variety defined by the vanishing of
\[ y^{\ell^n} - h^\ell = \prod_{r=0}^{\ell - 1} \bigl( y^{\ell^{n-1}} - \zeta^r h \bigr). \]
By contradiction ($C$ is geometrically irreducible by assumption), $S[0]$ cannot be empty.

For the second claim, notice that $\sum_{\xi \in S} n_\xi = n_\infty \equiv 0 \pmod{\ell}$. Since $\ell \mid n_\xi$ for each $\xi \in S - S[0]$, this implies that $0 \equiv \sum_{\xi \in S[0]} n_\xi \pmod{\ell}$. But each point $\xi \in S[0]$ has $\ell \nmid n_\xi$ by definition, and so there must be more than one point in $S[0]$.
\end{proof}

The tower of coverings \eqref{eqn:tower} also induces (canonically and contravariantly) a sequence of homomorphisms on the Jacobians of the curves $C_s$. For each $0 \leq s \leq n$, set $J_s := \Jac(C_s)$, and let $g_s := \dim J_s$. We have
\[ \begin{tikzcd}
\{0 \} = J_0 \arrow{r}{\pi_1^*} & J_1 \arrow{r}{\pi_2^*} & \cdots \arrow{r} & J_{n-1} \arrow{r}{\pi_n^*} & J_n = J.
\end{tikzcd} \]
We claim that each of these morphisms $\pi_s^*$ is a closed immersion. Observe that, as $S[0]$ is non-empty, $\Pi_s$ is totally ramified over at least one point for each $s \geq 1$. Thus, each morphism $\pi_s$ is totally ramified somewhere, and so cannot admit a nontrivial unramified subcover. That $\pi_s^*$ is a closed immersion now follows from the following more general fact.
\begin{proposition}\label{prop:closed_immersion}
Suppose $X$ and $Y$ are proper smooth curves over an algebraically closed field of characteristic zero. Suppose $\rho \colon Y \to X$ is a finite morphism, possibly branched. Let $\rho' \colon Y' \to X$ be the maximal abelian subcovering of $\rho$ which is everywhere unramified. Then the following are equivalent:
\begin{enumerate}[label=\emph{(\roman*)}, itemsep=1pt]
\item The induced morphism $\rho^* \colon \Jac(X) \to \Jac(Y)$ is a closed immersion.
\item The induced morphism on torsion, $\Jac(X)_{\mathrm{tors}} \to \Jac(Y)_{\mathrm{tors}}$, is injective.
\item $Y' = X$.
\end{enumerate}
\end{proposition}
\begin{proof}
The equivalence between (i) and (ii) is well-known. From \cite[III, Prop.~4.11]{Milne:1980} (see also the following discussion, pgs.~126--127), we have for any $N \geq 1$ the isomorphisms
\[ \Jac(X)[N] \cong \Het(X, \Z/N\Z(1)) \cong \Het(X, \bmu_N). \]
Here, $\Z/N\Z(1)$ denotes the Tate twist (recalled in \S6) by the cyclotomic character modulo $N$. Let ${}^\vee$ denote the Pontryagin dual. Taking a direct limit with respect to $N$, we obtain
\[ \Jac(X)_{\mathrm{tors}} = \Het(X, \Q/\Z(1)) = \Hom( \pi_1(X), \Q/\Z(1)) = \bigl(\pi_1(X)^\mathrm{ab} \bigr)^\vee. \]
Thus, (ii) is equivalent to $Y \to X$ inducing an inclusion
\[ \bigl(\pi_1(X)^\mathrm{ab} \bigr)^\vee \longrightarrow \bigl( \pi_1(Y)^\mathrm{ab} \bigr)^\vee. \]
But this is injective if and only if the homomorphism $\pi_1(Y)^\mathrm{ab} \to \pi_1(X)^\mathrm{ab}$ is surjective, and the cokernel of this final morphism is precisely $\Gal(Y'/X)$. Thus, condition (ii) is equivalent to $Y' = X$.
\end{proof}

Thus, for the remainder, we may and do view $J_{s-1}$ as an abelian subvariety of $J_s$, and define for each $s \geq 1$,
\[ A_s := J_s / J_{s-1}, \qquad h_s := \dim A_s = g_s - g_{s-1}. \]
\begin{proposition}\label{prop:l-torsion}
For each $1 \leq s \leq n$, $A_s[\ell] \cong J_s[\ell]/J_{s-1}[\ell]$.
\end{proposition}
\begin{proof}
To begin, we certainly have a filtration on $\ell$-torsion:
\[ \{0\} = J_0[\ell] \subseteq J_1[\ell] \subseteq \cdots \subseteq J_n[\ell] = J[\ell]. \]
Now, as abelian varieties are divisible as groups, the vertical morphisms in the following diagram are all surjective:
\[ \begin{tikzcd}
0 \arrow{r} & J_{s-1} \arrow{r} \arrow{d}[swap]{[\ell]} & J_s \arrow{r} \arrow{d}[swap]{[\ell]} & A_s \arrow{r} \arrow{d}[swap]{[\ell]} & 0 \\
0 \arrow{r} & J_{s-1} \arrow{r} & J_s \arrow{r} & A_s \arrow{r} & 0
\end{tikzcd}
\]
As the cokernel of $[\ell]$ is trivial, the Snake Lemma yields an exact sequence
\[ \begin{tikzcd}
0 \arrow{r} & J_{s-1}[\ell] \arrow{r} & J_s[\ell] \arrow{r} & A_s[\ell] \arrow{r} & 0,
\end{tikzcd} \]
and the claimed isomorphism holds.
\end{proof}

\begin{proposition}\label{prop:J_As_equiv}
Suppose that the variety $J$ has good reduction away from $\ell$. Let $\ten = \ten(k, \ell)$. Then the following are equivalent:
\begin{enumerate}[label=\emph{(\roman*)}, itemsep=1pt]
\item $k(J[\ell^\infty]) \subseteq \ten$.
\item $k(J[\ell]) \subseteq \ten$.
\item For every $1 \leq s \leq n$, $k(A_s[\ell]) \subseteq \ten$.
\end{enumerate}
\end{proposition}
\begin{proof}
The equivalence between (i) and (ii) follows from the reduction type of $J$, the results of Serre-Tate \cite{Serre-Tate:1968}, and the (group theoretic) fact that the kernel of the natural homomorphism $GL(T_\ell J) \to GL(J[\ell])$ is pro-$\ell$. For any abelian variety $B/k$, let $\rho_{B,\ell}$ denote the mod $\ell$ Galois representation attached to $B$. As established in \cite[Lemma 3.3]{Rasmussen-Tamagawa:2017}, the containment $k(B[\ell]) \subseteq \ten$ is equivalent to the condition that $\rho_{B,\ell}$ is upper triangular with powers of $\chi$, the $\ell$-cyclotomic character modulo $\ell$, appearing on the diagonal (or rather, may be placed in such form by appropriate choice of basis). The isomorphisms established in Proposition \ref{prop:l-torsion}
imply that $\rho_{J,\ell}$ has the form
\[ \rho_{J,\ell} \sim \begin{pmatrix}
   \rho_{A_1,\ell} & * & \cdots & * \\
         & \rho_{A_2,\ell} & \cdots & * \\
 & & \ddots & \vdots \\
 & & & \rho_{A_n,\ell} \end{pmatrix}.
\]
Thus, (iii) is equivalent to each $\rho_{A_s,\ell}$ being upper triangular with powers of $\chi$ on the diagonal. This in turn is equivalent to $\rho_{J,\ell}$ being upper triangular with powers of $\chi$ on the diagonal, which is in turn equivalent to (ii).
\end{proof}
In the remainder of this section, we compute the dimension of $A_s$ in terms of $\ell$ and $S[\less s]$. Let $\Gamma_s$ denote the Galois group of $\Pi_s$ after passing to the algebraic closure; that is,
\[ \Gamma_s := \Gal( \bar{C}_s/\bar{C}_0) \cong \Z/\ell^s \Z \]
for each $0 \leq s \leq n$. For any $0 \leq j < n$ and any $\xi \in S[j]$, we let $I_s(\xi)$ denote the inertia subgroup of $\Gamma_s$ attached to $\xi$ (viewed as a place of $\bar{C}_0$).
\begin{lemma}
For each $0 \leq j < n$ and each $\xi \in S[j]$,
\[ \# I_s(\xi) = \begin{cases} \ell^{s-j} & j < s \\ 1 & j \geq s \end{cases}.   \]
\end{lemma}
\begin{proof}
The order of $I_s(\xi)$ coincides with the ramification index in the subcovering $\bar{C}_s \to \bar{C}_0$. Recall that a model for $C_s$ is given by
\[ y_s^{\ell^s} = f(x) = \lambda \prod_{\xi' \in S} (x - \xi')^{n_{\xi'}}. \]
As $\xi \in S[j]$, $\ord_\ell n_\xi = j$. Thus, the point $\xi$ is unramified if and only if $s \leq j$. Otherwise, the ramification index must be $\ell^{s-j}$.
\end{proof}
\begin{lemma}\label{lemma:hs}
For any $1 \leq s \leq n$, the dimension of $A_s$ is given by
\[ h_s = \frac{1}{2} \ell^{s-1} (\ell - 1) \left( -2 + \sum_{j=0}^{s-1} r_j \right) = \frac{1}{2} \ell^{s-1}(\ell - 1) \bigl( \#S[\less s] - 2 \bigr). \]
\end{lemma}
\begin{proof}
The covering $\Pi_s$ branches only over $S[\less s]$, and the ramification index is $\ell^{s-j}$ for all points in the fibers over $S[j]$ when $j < s$. Moreover, each such fiber contains precisely $\ell^j = \frac{\ell^s}{\ell^{s-j}}$ points. Applying the Riemann-Hurwitz formula gives:
\[ 2g_s - 2 = -2 \ell^s + \sum_{j=0}^{s-1} r_j \ell^j (\ell^{s-j} - 1). \]
The final result follows directly from the fact that $h_s = g_s - g_{s-1}$.
\end{proof}

\section{The subspace $W_s$}\label{section:Ws}
In this section, we demonstrate an explicit $\F_\ell$-subspace $W_s \leq A_s[\ell]$ for each $s$, which we later use to more explicitly study the Galois representation $\rho_{A_s,\ell}$.

To begin, we establish a structural property of certain groups of divisors on a curve. Suppose $X/\bar{k}$ is a smooth projective curve. Recall that a divisor $D$ on $X$ is \emph{reduced} if $D \neq 0$ and $D$ is a finite sum of distinct prime divisors; i.e., $D = x_1 + x_2 + \cdots + x_n$ for $n$ distinct points $x_i \in X$. For such $D$, we have an equality of subgroups within $\Div X$,
\[ \Z x_1 + \cdots + \Z x_n = \Z D + \Z x_2 + \cdots + \Z x_n. \]
Note that both sides of the above equality represent internal direct sums within $\Div X$.
\begin{lemma}\label{lemma:torfree}
Let $D_1, \dots, D_n$ be reduced divisors on $X$ which, taken pairwise, have disjoint support. Let $\mathcal{W}$ be the subgroup generated by $D_1, \dots D_n$, and set $\mathcal{W}^0 := \mathcal{W} \cap \Div^0 X$.
\begin{enumerate}[label=(\alph*),itemsep=0pt]
\item $\mathcal{W}$ is a direct summand of $\Div X$.
\item The groups $\Div X / \mathcal{W}$, $\Div X / \mathcal{W}^0$, and $\Div^0 X / \mathcal{W}^0$ are torsion-free.
\end{enumerate}
\end{lemma}
\begin{proof}
We first show (a) implies (b). Assuming (a) holds, $\Div X / \mathcal{W}$ is isomorphic to a direct summand of the free $\Z$-module $\Div X$, and is therefore torsion-free. The natural map $\Div^0 X \to \Div X / \mathcal{W}$ has $\mathcal{W}^0$ as its kernel, and so induces an inclusion $\Div^0 X / \mathcal{W}^0 \hookrightarrow \Div X / \mathcal{W}$. Thus, $\Div^0 X / \mathcal{W}^0$ is torsion free. Now, there is a short exact sequence
\[
\begin{tikzcd}
0 \ar{r} & {\displaystyle \frac{\Div^0 X}{\mathcal{W}^0}} \ar{r} & {\displaystyle \frac{\Div X}{\mathcal{W}^0}} \arrow{r} & {\displaystyle \frac{\Div X}{\Div^0 X}} \ar{r} & 0
\end{tikzcd}, \]
whose first and third terms are torsion-free (the third being isomorphic to $\Z$). Thus $\Div X / \mathcal{W}^0$ is also torsion free.

For (a), let $Z \subseteq X$ be the support of $\sum D_i$, and let $m_i = \deg D_i$. Writing $D_i = \sum_{j=1}^{m_i} x_{i,j}$, we have $x_{i,j} = x_{i',j'}$ only when $i = i'$, $j = j'$, by assumption. Now we may express $\Div X$ via internal sums as
\[ \begin{split}
\Div X & = \sum_{x \in Z} \Z x + \sum_{x \not\in Z} \Z x = \sum_{i=1}^n \sum_{j=1}^{m_i} \Z x_{i,j} + \sum_{x \not\in Z} \Z x \\
 & = \sum_{i=1}^n \left( \Z D_i + \sum_{j=2}^{m_i} \Z x_{i,j} \right) + \sum_{x \not\in Z} \Z x \\
 & = \sum_{i=1}^n \Z D_i + \sum_{i=1}^n \sum_{j=2}^{m_i} \Z x_{i,j} + \sum_{x \not\in Z} \Z x \\
 & = \mathcal{W} + \sum_{i=1}^n \sum_{j=2}^{m_i} \Z x_{i,j} + \sum_{x \not\in Z} \Z x.
\end{split}
\]
The hypotheses guarantee that every internal sum in the above is direct, and so $\mathcal{W}$ is a direct summand of $\Div X$.
\end{proof}

For each $0 \leq s \leq n$, we define a mapping $[ \,\cdot\, ]_s \colon \P^1 \longrightarrow \Div C_s$ as follows. For $s=0$, set $[\xi]_0 := \xi$. (Recall, $C_0 = \P^1$.) For $s > 0$, we define
\[ [\xi]_s := \sum_{\eta \in \Pi_s^{-1}(\xi)} \eta. \]
That is, $[\xi]_s$ is the support divisor for the fiber of $\Pi_s$ over $\xi$, but with \emph{no} weighted coefficients coming from ramification. By definition, this is an effective reduced divisor satisfying
\[ \deg \, [\xi]_s = \begin{cases} \ell^s & \xi \not\in S[\less s] \\ \ell^j & \xi \in S[j], \qquad j < s \end{cases}. \]
Because $\Pi_s$ is $k$-rational, for any $\xi$ and any $\sigma \in G_k$, $[\sigma(\xi)]_s = \sigma \cdot [\xi]_s$. (That is, $[\,\cdot\,]$ is $G_k$-equivariant.)
\begin{lemma}\label{lemma:pullback_formula}
Let $1 \leq s \leq n$, and let $\xi \in \P^1$.
\begin{enumerate}[label=\emph{(\alph*)}]
\item If $\xi \not\in S[\less s]$, then $\pi_s^*([\xi]_{s-1}) = [\xi]_s$.
\item If $\xi \in S[\less s]$, then $\pi_s^*([\xi]_{s-1}) = \ell \cdot [\xi]_s$.
\item If $\xi \in S[j]$ for $0 \leq j \leq s$ and $j \neq n$, then $\Pi_s^* (\xi) = \ell^{s-j}[\xi]_s$.
\item If $\xi \not\in S[\less s]$, then $\Pi_s^*(\xi) = [\xi]_s$.
\end{enumerate}
\end{lemma}
\begin{proof}
By definition, the degree $\ell$ covering $\pi_s \colon C_s \to C_{s-1}$ is, possibly after base change to a finite extension $k'/k$, Galois with Galois group $\Z/\ell \Z$. Thus, every point of $C_s(\bar{k})$ is either unramified or totally ramified under $\pi_s$. The stratification of $S$ introduced in \S \ref{sec:S} guarantees the following: for any $\xi \in \P^1$, $\pi_s$ is totally ramified over every point in the support of $[\xi]_{s-1}$ if $\xi \in S[\less s]$, which demonstrates (a). Likewise, if $\xi \not\in S[\less s]$, then $\pi_s$ is unramified over every point in the support of $[\xi]_{s-1}$, which demonstrates (b).

We first prove (c) in case $j < s$. This follows from (a) and (b), since for $\xi \in S[j]$ we have
\[ \begin{split}
\Pi_s^*(\xi) & = \Pi_s^*([\xi]_0) \\
 & = \left( \pi_s^* \circ \cdots \circ \pi_1^* \right)([\xi]_0) \\
 & = \left( \pi_s^* \circ \cdots \circ \pi_{j+1}^* \right)([\xi]_j) \\
 & = \ell^{s-j} [\xi]_s.
\end{split}  \]
Similarly, (d) follows from (a), since $\xi \not\in S[\less s]$ implies $\xi \not\in S[\less t]$ for all $t \leq s$. Thus,
\[ \Pi_s^*\left( [\xi]_0 \right) = (\pi_s^* \circ \cdots \circ \pi_1^*)( [\xi]_0 ) = [\xi]_s. \]
Finally, (c) in the case $j = s$ follows from (d), completing the proof.
\end{proof}
We now fix a point $\xi_0 \in S[0]$ (this set is nonempty by Lemma \ref{lemma:S0-nonempty}).
\begin{corollary}\label{cor:pullback}
If $\xi \in \P^1(\bar{k})$ and $\xi \not\in S[\less s]$, then
\[ \Pi_s^* \bigl( [\xi]_0 - [\xi_0]_0 \bigr) = [\xi]_s - \ell^s [\xi_0]_s. \]
\end{corollary}
\begin{proof}
This follows immediately from Lemma \ref{lemma:pullback_formula} (c) and (d).
\end{proof}
We define the following subgroups of $\Div C_s$:
\[ \begin{split}
\mathcal{W}_s & := \bigl\langle [\xi]_s : \xi \in S[\less s] \bigr\rangle, \\
\mathcal{W}_s^0 & := \mathcal{W}_s \cap \Div^0 C_s, \\
\mathcal{V}_s & := \bigl\langle [\xi]_s : \xi \in S[j] \text{ where } 0 \leq j \leq s, j \neq n \bigr\rangle, \\
\mathcal{V}_s^0 & := \mathcal{V}_s \cap \Div^0 C_s.
\end{split} \]
For any $j$, $s$ satisfying $0 \leq j \leq s \leq n$, $j \neq n$, and $\xi \in S[j]$, set
\[ D_{s, \xi} := [\xi]_s - \ell^j [\xi_0]_s \in \Div^0 C_s. \]
Notice that $D_{s,\xi_0} = 0$.
\begin{lemma}\label{lemma:Vs0_free}
For any $1 \leq s \leq n$, $\mathcal{V}_s^0$ is free with $\Z$-basis
\[ \left\{ D_{s,\xi} : 0 \leq j \leq s, j \neq n, \xi \in S[j], \xi \neq \xi_0 \right\}. \]
Moreover, $\Div^0 C_s / \mathcal{V}_s^0$ is torsion-free.
\end{lemma}
\begin{proof}
Suppose $D \in \mathcal{V}_s^0$. A priori, $D$ has the form $\sum_\xi b_\xi [\xi]_s$, and so also
\[ D = \sum_{\xi \,\neq\, \xi_0} b_\xi D_{s,\xi} + c[\xi_0]_s, \qquad \text{ for some } c \in \Z. \]
But $D$, $D_{s,\xi}$ are degree $0$, and so $c = 0$. Thus, the given set spans $\mathcal{V}_s^0$.

Suppose $a_\xi \in \Z$ are chosen such that $\sum_{\xi \neq \xi_0} a_\xi D_{s,\xi} = 0$. Let $\eta_0 \in C_s(\bar{k})$ be the unique point in $\Pi^{-1}_s(\xi_0)$. By definition of $D_{s,\xi}$, there exists $b \in \Z$ such that
\[ \sum_{\xi \,\neq\, \xi_0^{\vphantom{-1}}} \; \sum_{\eta \,\in\, \Pi^{-1}_s(\xi)}  a_\xi \cdot \eta \;\; + b \cdot \eta_0 = 0. \]
But each point of $C_s$ appears at most once in the left hand expression, and it may only vanish if $a_\xi = 0$ for all $\xi$. This demonstrates the given set is a basis for $\mathcal{V}_s^0$. The final claim follows from Lemma \ref{lemma:torfree}.
\end{proof}
By essentially the same argument we have:
\begin{lemma}\label{lemma:W0_basis}
For any $1 \leq s \leq n$, $\mathcal{W}_s^0$ is free with $\Z$-basis
\[ \bigl\{ D_{s,\xi} : \xi \in S[\less s] - \{\xi_0\} \bigr\}, \]
and $\Div^0 C_s / \mathcal{W}_s^0$ is torsion-free.
\end{lemma}

\begin{lemma}\label{lemma:Ds_xi}
Suppose $1 \leq s \leq n$, $0 \leq j \leq s$, and $j \neq n$. Suppose $\xi \in S[j]$.
\begin{enumerate}[itemsep=1pt, label=(\alph*)]
\item If $j < s$, then $\ell D_{s,\xi} = \pi_s^* D_{s-1,\xi}$.
\item $\ell^{s-j} D_{s,\xi}$ is principal.
\end{enumerate}
\end{lemma}
\begin{proof}
For (a), we have $j < s$  and so appealing to Lemma \ref{lemma:pullback_formula} yields
\[ \ell D_{s,\xi} = \ell[\xi]_s - \ell^{j+1} [\xi_0]_s = \pi_s^* [\xi]_{s-1} - \ell^j \pi_s^* [\xi_0]_{s-1} = \pi_s^* D_{s-1,\xi}.\]
For (b), Lemma \ref{lemma:pullback_formula}(c) applies, and
\[ \ell^{s-j} D_{s,\xi} = \ell^{s-j} [\xi]_s - \ell^s [\xi_0]_s = \Pi_s^* ( [\xi]_0 - [\xi_0]_0).\]
But any degree $0$ divisor on $\P^1$ is principal. Since the pullback of a principal divisor is principal, the result holds.
\end{proof}
We define $\alpha_s \colon \mathcal{W}_s^0 \to A_s$ by the following composition, where the latter two surjective maps are the canonical projections:
\[
\begin{tikzcd}[column sep = large]
\mathcal{W}^0_s \arrow[r, hook] & \Div^0 C_s \arrow[->>]{r} & J_s \arrow[->>]{r} & A_s
\end{tikzcd}
\]
We let $W_s$ denote the image $\alpha_s(\mathcal{W}^0_s)$, and set $\Delta_{s,\xi} := \alpha_s(D_{s,\xi})$.
\begin{lemma}\label{lemma:Ws_in_Asl}
$W_s \leq A_s[\ell]$.
\end{lemma}
\begin{proof}
Suppose $\eta \in \Pi_s^{-1}( S[\less s])$. Necessarily, the degree $\ell$ morphism $\pi_s$ is totally ramified at $\eta$. At the level of divisors, we have
\[ \pi_s^* \pi_{s,*} (\eta) = \ell \cdot \eta. \]
For any $\xi \in S[\less s]$, $[\xi]_s$ is a formal sum of such points $\eta$, and so for any $\xi \in S[\less s]$, we have $\pi_s^* \pi_{s,*} ( [\xi]_s ) = \ell[\xi]_s$. Thus $\pi_s^* \pi_{s,*}$ coincides with the multiplication-by-$\ell$ map on $\mathcal{W}_s$, and this condition continues to hold after passing from the subgroup $\mathcal{W}_s^0$ to $J_s$.
In other words, the following diagram is commutative, where the vertical arrows are the natural projections.
\[ \begin{tikzcd}[column sep = large]
\mathcal{W}_s^0 \arrow{rr}{[\ell]} \arrow[d] & & \mathcal{W}_s^0 \arrow[d] \\
J_s \arrow{r}{\pi_{s,*}} \arrow[d, two heads] & J_{s-1} \arrow[r, hook, "\pi_s^*"] & J_s \arrow[d, two heads] \\
A_s \arrow{rr}[swap]{[0]} & & A_s
\end{tikzcd}
\]
It suffices to show that any $\Delta \in W_s$ is killed by $[\ell]$. By definition, there exists $D \in \mathcal{W}_s^0$ such that $\Delta = \alpha_s(D)$. At the level of the Jacobian varieties, the map $\pi_s^*$ is a closed immersion (see the discussion prior to Proposition \ref{prop:closed_immersion}). This implies that the image of $[\ell]D$ in $J_s$ actually lies in the subvariety $J_{s-1}$. But $J_{s-1}$ is precisely the kernel of the projection $J_s \twoheadrightarrow A_s$, and so $[\ell] \Delta = 0$.
\end{proof}
Notice that $A_s[\ell]$ carries the structure of a vector space over $\F_\ell$, and it is immediately clear that $W_s$ is a subspace.
\begin{lemma}\label{lemma:Ws_Gk_stable}
Suppose $1 \leq s \leq n$. The subspace $W_s$ is $G_k$-stable. Moreover, $k(W_s) \subseteq k(S[\less s])$.
\end{lemma}
\begin{proof}
Recall that the morphisms $\Pi_s$ are $k$-rational and that $S[\less s]$ is $G_k$-stable. The set of divisors $\{ [\xi]_s : \xi \in S[\less s] \}$ is permuted by the action of $G_k$. It follows that the subgroup $\mathcal{W}_s$ of $\Div C_s$ is $G_k$-stable. Thus, $\mathcal{W}^0_s$ is $G_k$-stable, as it is the intersection of two $G_k$-stable subgroups. Now $W_s$ must be $G_k$-stable, as it is the image of a $G_k$-stable set under the $G_k$-equivariant map $\alpha_s$.

Suppose $\sigma \in G_k$ fixes $S[\less s]$ pointwise. Then for any $\xi \in S[\less s]$, the divisor $[\xi]_s$ is fixed by $\sigma$. Thus the generators $D_{s,\xi}$ of $\mathcal{W}^0_s$ are also fixed by $\sigma$. Now the $G_k$-equivariance of $\alpha_s$ demonstrates that $W_s$ is fixed pointwise by $\sigma$, and so $k(W_s) \subseteq k(S[\less s])$. \end{proof}

\section{The kernel of $\alpha_s$}

In this section, we determine the kernel of $\alpha_s$. This will allow us to determine the $\F_\ell$-dimension of $W_s$. It will also be useful in the later sections, when we consider the extension $k(S[\less s])/k(W_s)$ more thoroughly. In light of Lemma \ref{lemma:Ws_Gk_stable}, it is natural to ask if this extension is trivial, particularly in the case $s = n$ (i.e., $S[\less s] = S$). The answer is usually yes, but there are some exceptional cases (even beyond some trivial cases where $C$ has genus zero).

For any $\xi \in S[j]$, set $u_\xi := \ell^{-j} n_\xi$. By definition of $S[j]$, $u_\xi$ is an integer which is prime to $\ell$. For any $1 \leq s \leq n$, we define a divisor on $C_s$,
\begin{equation}\label{eqn:Rs_defn}
\mathfrak{R}_s := \sum_{j=0}^{s-1} \sum_{\xi \in S[j]} u_\xi D_{s,\xi} \in \mathcal{W}^0_s.
\end{equation}
\begin{lemma}\label{lemma:Rs_principal}
The divisor $\mathfrak{R}_s$ is principal.
\end{lemma}
\begin{proof}
Recall that $n_\infty - \sum_{\xi \in S} n_{\xi} = 0$, by \eqref{eqn:n_infty}. Consequently, for \emph{any} divisor $E$:
\begin{equation}\label{eqn:trivial_divisor}
0 = \left( \sum_{j=0}^{s-1} \sum_{\xi \in S[j]} n_\xi \right)E + \left( \sum_{j=s}^{n-1} \sum_{\xi \in S[j]} n_\xi \right)E - n_\infty E.
\end{equation}
View $f$ as a morphism $\P^1 \to \P^1$ and define $f_s := f \circ \Pi_s \colon C_s \to \P^1$. As an element of $k(C_s) = k(x, y_s)$, $f_s$ coincides with $y_s^{\ell^s}$. By Lemma \ref{lemma:pullback_formula},
\[ \begin{split}
\divisor_{C_s} y_s & = \ell^{-s} \cdot \divisor_{C_s} f_s = \ell^{-s} \Pi_s^* (\divisor_{\P^1} f) \\
 & = \ell^{-s} \Pi_s^* \bigl( \sum_{\xi \in S} n_\xi \cdot \xi - n_\infty \cdot \infty \bigr) \\
 & = \ell^{-s} \bigl( \sum_{j=0}^{s-1} \sum_{\xi \in S[j]} n_\xi \ell^{s-j} [\xi]_s + \sum_{j=s}^{n-1} \sum_{\xi \in S[j]} n_\xi \cdot [\xi]_s - n_\infty \cdot [\infty]_s \bigr). \\
\end{split} \]
Now, taking $E = [\xi_0]_s$ and ``adding $0$'' via \eqref{eqn:trivial_divisor}, we obtain
\[
\begin{split}
\divisor_{C_s} y_s & = \sum_{j=0}^{s-1} \sum_{\xi \in S[j]} u_\xi \left( [\xi]_s - \ell^j [\xi_0]_s \right) \\
& {} \qquad \qquad + \sum_{j=s}^{n-1} \sum_{\xi \in S[j]} \frac{n_\xi}{\ell^s} \left( [\xi]_s - \ell^s [\xi_0]_s \right) - \frac{n_\infty}{\ell^s} \left( [\infty]_s - \ell^s [\xi_0]_s \right). \\
& = \mathfrak{R}_s + \sum_{j=s}^{n-1} \sum_{\xi \in S[j]} \frac{n_\xi}{\ell^s} \Pi_s^* \left( [\xi]_0 - [\xi_0]_0 \right) - \frac{n_\infty}{\ell^s} \Pi_s^* \left( [\infty]_0 - [\xi_0]_0 \right) \\
& = \mathfrak{R}_s + \Pi_s^* Y,
\end{split}
\]
where $Y \in \Div^0 \P^1$. But every degree $0$ divisor on $\P^1$ is principal, and the result follows.
\end{proof}

\begin{proposition}
For any $1 \leq s \leq n$, $\ker \alpha_s = \langle \mathfrak{R}_s \rangle + \ell \mathcal{W}_s^0$. That is, the following sequence is exact:
\begin{equation}\label{eqn:alpha_s_seq}
\begin{tikzcd}[column sep = 1cm]
0 \ar{r} & \langle \mathfrak{R}_s \rangle + \ell \mathcal{W}^0_s \ar{r}{\iota} & \mathcal{W}^0_s \ar{r}{\alpha_s} & W_s \ar{r} & 0
\end{tikzcd} \end{equation}
\end{proposition}
\begin{proof}
That $\langle \mathfrak{R}_s \rangle + \ell \mathcal{W}^0_s \leq \ker \alpha_s$ follows immediately from Lemmas \ref{lemma:Ws_in_Asl} and \ref{lemma:Rs_principal}. For the reverse containment, suppose $Z \in \mathcal{W}^0_s$ satisfies $\alpha_s(Z) = 0$. By the definition of $\alpha_s$ the divisor $Z$ is, up to some principal divisor on $C_s$, an element of $J_{s-1}$. In other words, there must exist $h \in \bar{k}(C_s)^\times$ and $D \in \Div^0 C_{s-1}$ such that $Z = \divisor_{C_s} h + \pi_s^* D$.

Choose $m_\xi \in \Z$ such that $Z = \sum_{j<s} \sum_{\xi \in S[j]} m_\xi D_{s,\xi}$, and set
\[ E := \sum_{j=0}^{s-1} \sum_{\xi \in S[j]} m_\xi D_{s-1,\xi} - \ell D. \]
By Lemma \ref{lemma:Ds_xi} (a),
\[ \begin{split}
\divisor_{C_s} h^\ell = \ell (Z - \pi_s^* D) & = \sum_{j=0}^{s-1} \sum_{\xi \in S[j]} m_\xi \ell D_{s,\xi} - \ell \pi_s^* D \\
& = \sum_{j=0}^{s-1} \sum_{\xi \in S[j]} m_\xi \pi_s^* D_{s-1,\xi} - \ell \pi_s^*D = \pi_s^* E.
\end{split} \]
Consequently, $\pi_s^*E \equiv 0$ in $J_s$. But $\pi_s^* \colon J_{s-1} \to J_s$ is an immersion and so $E \equiv 0$ in $J_{s-1}$. Thus, there exists $q \in \bar{k}(C_{s-1})$ such that $\divisor_{C_{s-1}} q = E$. There is some $\gamma \in \bar{k}^\times$ such that $h^\ell = \gamma \cdot \pi_s^*q$, which demonstrates that
\[ (\pi_s^* q)^{\frac{1}{\ell}} \in \bar{k}(C_s) = \bar{k}(C_{s-1})(y_{s-1}^{\frac{1}{\ell}}), \]
and the standard facts of Kummer theory (e.g., \cite[\S4.7, pg.~58]{Koch:ANT}) guarantee the existence of an integer $0 \leq i \leq \ell - 1$ and a function $r \in \bar{k}(C_{s-1})^\times$ such that $q = y_{s-1}^i \cdot r^\ell$. In terms of divisors,
\[ \divisor_{C_{s-1}} q = i \divisor_{C_{s-1}} y_{s-1} + \ell \divisor_{C_{s-1}} r. \]
Observing
\[ \begin{split}
\divisor_{C_s} h^\ell & = \pi_s^* \divisor_{C_{s-1}} q \\
 & = i \pi_s^* \divisor_{C_{s-1}} y_{s-1} + \ell \pi_s^* \divisor_{C_{s-1}} r \\
 & = i \ell \divisor_{C_s} y_s + \ell \pi_s^* \divisor_{C_{s-1}} r,
\end{split} \]
we conclude $\divisor_{C_s} h = i \divisor_{C_s} y_s + \pi_s^* \divisor_{C_{s-1}} r$. Already in the proof of Lemma \ref{lemma:Rs_principal}, we see $\divisor_{C_s} y_s = \mathfrak{R}_s + \Pi_s^* \divisor_{\P^1} g_0$ for some $g_0 \in \bar{k}(\P^1)$. Thus
\[ \begin{split}
Z & = i \divisor_{C_s} y_s + \pi_s^* \divisor_{C_{s-1}} r + \pi_s^*D \\
  & = i \mathfrak{R}_s + i \Pi_s^* \divisor_{\P^1} g_0 + \pi_s^* \divisor_{C_{s-1}} r + \pi_s^* D \\
  & = i \mathfrak{R}_s + \pi_s^* (D + \divisor_{C_{s-1}} g)
\end{split} \]
for some $g \in \bar{k}(C_{s-1})$. On the other hand, $Z - i \mathfrak{R}_s \in \mathcal{W}_s^0$, and so there exist integers $b_\xi$ such that
\[ \pi_s^* (D + \divisor_{C_{s-1}} g)  = \sum b_\xi D_{s,\xi}, \]
with the sum running over $\xi \in S[\less s] - \{\xi_0\}$. Appealing to Lemma \ref{lemma:Ds_xi} however, we have
\[ \ell \pi_s^* (D + \divisor_{C_{s-1}} g) = \sum \ell b_\xi D_{s,\xi} = \sum b_\xi \pi_s^* D_{s-1,\xi}.
\]
As $\pi_s^*$ is injective on divisors, we see $\sum b_\xi D_{s-1,\xi} \in \ell \Div^0 C_{s-1}$. By Lemma \ref{lemma:Vs0_free}, $\ell \mid b_\xi$ for every $\xi$, and so $Z \in \langle \mathfrak{R}_s \rangle + \ell \mathcal{W}_s^0$, completing the proof.
\end{proof}

\section{$W_s$ as an $\F_\ell[G_k]$-module}

We now consider the structure of $W_s$ as an $\F_\ell[G_k]$-module more carefully. Let $X/k$ be a curve. For any $E \in \Div X$, we write $\overline{E}$ for the element $E \otimes 1$ in $\Div X \otimes \F_\ell$. Notice that $\Div X \otimes \F_\ell$ carries the structure of an $\F_\ell[G_k]$-module. Moreover, if $M \leq \Div X$ is any $G_k$-stable subgroup satisfying $((\Div X)/M)[\ell] = \{0 \}$, then $M \otimes \F_\ell$ is naturally an $\F_\ell[G_k]$-submodule of $\Div X \otimes \F_\ell$. Thus, $\mathcal{W}_s^0 \otimes \F_\ell$ is an $\F_\ell[G_k]$-submodule of $\Div C_s \otimes \F_\ell$, by Lemma \ref{lemma:W0_basis}.
\begin{lemma}\label{lemma:WsFl}
For any $1 \leq s \leq n$, $k(S[\less s]) = k(\mathcal{W}_s \otimes \F_\ell)$.
\end{lemma}
\begin{proof}
We introduce $\F_\ell[G_k]$-modules $M[j]$ and $M[\less s]$, defined as
\[ M[j] := \bigoplus_{\xi \in S[j]} \F_\ell \xi, \qquad \qquad M[\less s] := \bigoplus_{j=0}^{s-1} M[j]. \]
The action of $G_k$ on $M[j]$ is induced from that on $S[j]$ in the obvious way: $\sigma \cdot (\sum a_\xi \cdot \xi) = \sum a_\xi \cdot \sigma(\xi)$. Clearly $k(S[\less s]) = k(M[\less s])$. As there is also an isomorphism of $\F_\ell[G_k]$-modules,
\[ M[\less s] \cong \mathcal{W}_s \otimes \F_\ell, \qquad \xi \mapsto \overline{[\xi]}_s, \]
it follows that $k(S[\less s]) = k(\mathcal{W}_s \otimes \F_\ell)$.
\end{proof}
\begin{lemma}\label{lemma:Rsbar}
Suppose $1 \leq s \leq n$. In $\Div C_s \otimes \F_\ell$, we have
\begin{equation}\label{eqn:Rsbar}
\overline{\mathfrak{R}}_s = \sum_{j=0}^{s-1} \sum_{\xi \in S[j]} u_\xi \overline{[\xi]}_s, \qquad u_\xi \not\equiv 0 \pmod{\ell}.
\end{equation}
Moreover, $\overline{\mathfrak{R}}_s \in (\mathcal{W}_s^0 \otimes \F_\ell)^{G_k}$.
\end{lemma}
\begin{proof}
First, suppose that \eqref{eqn:Rsbar} holds. By \eqref{eqn:Rs_defn}, $\overline{\mathfrak{R}}_s$ clearly lies in $\mathcal{W}_s^0 \otimes \F_\ell$. Because the polynomial $f(x)$ appearing in the affine model of $C$ is defined over $k$, the value $n_\xi$ is constant as $\xi$ runs over any $G_k$-orbit within the roots of $f(x)$. Moreover, each such orbit is a subset of $S[j]$ for some $j$, and so the coefficients $u_\xi$ are also constant over each orbit. Thus, $\overline{\mathfrak{R}}_s$ is fixed by the action of $G_k$.

It remains to show \eqref{eqn:Rsbar}. By the definition of $D_{s,\xi}$, we have in $\Div C_s$:
\[ \mathfrak{R}_s = \sum_{j=0}^{s-1} \sum_{\xi \in S[j]} \frac{n_\xi}{\ell^j} [\xi]_s - \left( \sum_{j=0}^{s-1} \sum_{\xi \in S[j]} n_\xi \right) \cdot [\xi_0]_s. \]
Recall that $n_\infty = \sum_{\xi \in S} n_\xi$ and (because we have arranged $\infty$ to lie outside the branch locus) $n_\infty \equiv 0 \pmod{\ell^n}$. By definition, $\ell^j \mid n_\xi$ if $\xi \in S[j]$. Consequently,
\[ \sum_{j=0}^{s-1} \sum_{\xi \in S[j]} n_\xi = n_\infty - \sum_{j=s}^n \sum_{\xi \in S[j]} n_\xi \]
is divisible by $\ell^s$ and $\overline{\mathfrak{R}}_s$ has the claimed form.
\end{proof}
\begin{lemma}\label{lemma:Ws_FlGk}
As $\F_\ell[G_k]$-modules, $W_s \cong (\mathcal{W}_s^0 \otimes \F_\ell)/\langle \overline{\mathfrak{R}}_s \rangle$.
\end{lemma}
\begin{proof}
We apply the functor $- \otimes \F_\ell$ to \eqref{eqn:alpha_s_seq} to obtain
\[ \begin{tikzcd}[column sep = large]
\langle \overline{\mathfrak{R}}_s \rangle \ar{r}{\iota \otimes 1} & \mathcal{W}^0_s \otimes \F_\ell \ar{r}{\alpha_s \otimes 1} & W_s \ar{r} & 0.
\end{tikzcd} \]
(Since $W_s$ is killed by $\ell$, $W_s \cong W_s \otimes \F_\ell$.) Moreover, $\iota \otimes 1$ is injective upon inspection, and so the resulting sequence is exact.
\end{proof}
Since $\{ \overline{D}_{s,\xi} : \xi \in S[\less s] - \{\xi_0\} \}$ gives an $\F_\ell$-basis for $\mathcal{W}_s^0 \otimes \F_\ell$, and $\overline{\mathfrak{R}}_s \neq 0$ by Lemma \ref{lemma:Rsbar}, we conclude:
\begin{corollary}
$\dim_{\F_\ell} W_s = \#S[\less s] - 2$.
\end{corollary}

\section{Galois action on $A_s[\ell]$}

In this section, we demonstrate that $W_s$ lies in a chain of $G_k$-stable subspaces of $A_s[\ell]$, and use this to give an explicit description of the mod $\ell$ Galois representation on $A_s$.

For any $1 \leq s \leq n$, let $\Delta_s := \Gal \left( \bar{k}(C_s)/\bar{k}(C_0) \right)$. This is a cyclic group of order $\ell^s$; we let $\delta_s$ denote a fixed generator for this group.
\begin{lemma}\label{lemma:action_of_delta_s}
Suppose $\xi \in S[j]$, where $j < s$. Then $D_{s,\xi}$ is $\Delta_s$-stable.
\end{lemma}
\begin{proof}
Recall that $D_{s,\xi} = [\xi]_s - \ell^j[\xi_0]_s$, and $[\xi]_s = \sum \eta$, where the sum runs over the points in the fiber $\Pi_s^{-1}(\xi)$. Thus, $[\xi]_s$ and $[\xi_0]_s$ are $\Delta_s$-stable, and  $D_{s,\xi}$ is also.
\end{proof}

Let us fix $\zeta_s \in \bmu_{\ell^s}$, a primitive $\ell^s$th root of unity.  Set
\[ \Phi_s := \frac{X^{\ell^s} - 1}{X^{\ell^{s-1}}-1} = \sum_{j=0}^{\ell-1} X^{j \cdot \ell^{s-1}} \in \Z[X], \]
the $\ell^s$th cyclotomic polynomial. Observe that as rings
\[ \frac{\Z[\Delta_s]}{(\Phi_s(\delta_s))} \cong \Z[\zeta_s], \qquad \text{ via } \quad \delta_s \mapsto \zeta_s. \]
The inclusion $\Delta_s \hookrightarrow \Aut C_s$ induces a $G_k$-equivariant ring homomorphism $\Z[\Delta_s] \to \End J_s$.
\begin{lemma}
Viewing $\Phi_s(\delta_s)$ as an endomorphism of $J_s$, $\Phi_s(\delta_s)(J_s) \subseteq J_{s-1}$.
\end{lemma}
\begin{proof}
Observe that the subgroup $\Gamma_s := \Gal \left( \bar{k}(C_s)/\bar{k}(C_{s-1}) \right) \leq \Delta_s$ is cyclic of order $\ell$, generated by $\delta_s^{\ell^{s-1}}$. Thus,
\[ \Phi_s(\delta_s) = \sum_{j=0}^{\ell-1} \left( \delta_s^{\ell^{s-1}} \right)^j = \sum_{\sigma \in \Gamma_s} \sigma. \]
For any point $z \in C_s$, we have the following equality of divisors:
\[ \Phi_s(\delta_s)(z) = \sum_{\sigma \in \Gamma_s} z^\sigma = \pi_s^* \pi_{s,*}(z). \]
However, $\pi_s^* \colon J_{s-1} \to J_s$ is a closed immersion and so the result follows.
\end{proof}
Consequently, $\Phi_s(\delta_s)$ acts trivially on $A_s$, and so the action of $\Z[\Delta_s]$ on $J_s$ induces a well-defined action of $\Z[\Delta_s]/(\Phi_s(\delta_s))$ on $A_s$. Let $T_s := T_\ell(A_s)$, the $\ell$-adic Tate module of $A_s$. The ring
\[ \Z_\ell[\zeta_s] := \Z[\zeta_s] \otimes_\Z \Z_\ell \]
inherits an action on $T_s$ from the isomorphism $\Z[\zeta_s] \cong \Z[\Delta_s]/(\Phi_s(\delta_s))$. Recall that $\Z_\ell[\zeta_s]$ is a discrete valuation ring whose maximal ideal $\mathfrak{m}$ is generated by the element $\varepsilon := \zeta_s - 1$. In this ring, the principal ideal generated by $\ell$ is totally ramified:
\[ \ell \Z_\ell[\zeta_s] = (\varepsilon \Z_\ell[\zeta_s])^{ \varphi(\ell^s) }, \]
where $\varphi$ denotes Euler's totient function.

It is well-known that $T_s$ is a free $\Z_\ell[\zeta_s]$-module of finite rank; we set $m := \rank_{\Z_\ell[\zeta_s]} T_s$. Since $\rank_{\Z_\ell} \Z_\ell[\zeta_s] = \varphi(\ell^s)$, from Lemma \ref{lemma:hs} we conclude
\[ m = -2 + \sum_{j=0}^{s-1} r_j = \dim_{\F_\ell} W_s. \]
As $A_s[\ell] \cong T_s \otimes_{\Z_\ell} \F_\ell$, we obtain:
\begin{corollary}
The $(\Z_\ell[\zeta_s] \otimes_\Z \F_\ell)$-module $A_s[\ell]$ is free of rank $m$.
\end{corollary}
We now consider the following chain of ideals in $\Z_\ell[\zeta_s]$ generated by the powers of $\varepsilon$:
\begin{equation}\label{eqn:chain}
(\ell) = \left( \varepsilon^{\varphi(\ell^s)} \right) \subseteq \left( \varepsilon^{\varphi(\ell^s)-1} \right) \subseteq \cdots \subseteq (\varepsilon) \subseteq \left( \varepsilon^0 \right) = \Z_\ell[\zeta_s].
\end{equation}
Notice that for any integer $0 \leq i < \varphi(\ell^s)$, we have isomorphisms of $\F_\ell$-modules:
\[ \F_\ell \cong \frac{\Z_\ell[\zeta_s]}{\mathfrak{m}} \cong \frac{ (\varepsilon^i) }{ (\varepsilon^{i+1}) }, \]
where the second isomorphism is given by $1 \mapsto \varepsilon^i$.

For any $\Z_\ell[\zeta_s] \otimes_\Z \F_\ell$-module $V$ and any $\beta \in \Z_\ell[\zeta_s]$, we let $V[\beta]$ denote the submodule killed by multiplication-by-$\beta$.
\begin{lemma}\label{lemma:Vepsilon}
Let $V$ be a finite free $\Z_\ell[\zeta_s] \otimes_\Z \F_\ell$-module, and let $a, b$ be nonnegative integers satisfying $a + b = \varphi(\ell^s)$. Then
\begin{enumerate}[label=(\alph*)]
\item $\varepsilon^a V = V[\varepsilon^b]$.
\item For $b > 0$, the homomorphism
\[ \varepsilon^a V \longrightarrow \varepsilon^{\varphi(\ell^s)-1}V = V[\varepsilon], \qquad \qquad v \mapsto \varepsilon^{b-1} v, \]
induces an isomorphism $\varepsilon^a V / \varepsilon^{a+1} V \cong V[\varepsilon]$.
\item As $\F_\ell$-vector spaces, $\dim_{\F_\ell} V = \varphi(\ell^s) \cdot \dim_{\F_\ell} V[\varepsilon]$.
\end{enumerate}
\end{lemma}
\begin{proof}
Statement (a) is immediate. Statement (b) follows from noting the given map is surjective with kernel $\varepsilon^{a+1}V$. Now (c) follows from (b) when we consider the filtration on $V$ induced by \eqref{eqn:chain}:
\[ 0 \leq \varepsilon^{\varphi(\ell^s)-1} V \leq \cdots \leq \varepsilon^2 V \leq \varepsilon V \leq V. \qedhere \]
\end{proof}
\begin{proposition}
$W_s = A_s[\varepsilon]$.
\end{proposition}
\begin{proof}
Both $W_s$ and $A_s[\varepsilon]$ are $\F_\ell$-subspaces of $A_s[\ell]$. Applying Lemma \ref{lemma:Vepsilon} to the vector space $V=A_s[\ell]$, we find $W_s$ and $A_s[\varepsilon]$ both have dimension $m$. Thus, it will suffice to demonstrate that every generator $\Delta_{s,\xi}$ of $W_s$ lies in $A_s[\varepsilon]$. By Lemma \ref{lemma:action_of_delta_s}, we see $\Delta_{s,\xi}$ is killed by $(\delta_s - 1)$. Consequently,
\[ \varepsilon \cdot \Delta_{s,\xi} = (\zeta_s - 1) \cdot \Delta_{s,\xi} = \alpha_s \bigl( (\delta_s - 1) D_{s,\xi} \bigr) = 0. \]
Thus, $\Delta_{s,\xi} \in A_s[\varepsilon]$, and so $W_s = A_s[\varepsilon]$ as claimed.
\end{proof}

In preparation for the next result, we briefly recall the notion of the Tate twist of a $G_k$-module (see also \cite{Shatz:1986}). Set $\bmu := \bmu(\bar{k})$, the group of all roots of unity of $\bar{k}$. Once and for all, choose $\mathbf{\upzeta} := \{ \zeta_n \}_{n=1}^\infty$, a compatible system of elements $\zeta_n \in \bmu$, where $\zeta_n$ is a primitive $\ell^n$th root of unity and $\zeta_{n+1}^\ell = \zeta_n$ for all $n \geq 1$. The $\ell$-adic Tate module $T_\ell \bmu$ is a free $\Z_\ell$-module of rank one, generated by $\mathbf{\upzeta}$. In fact, $T_\ell \bmu$ is a $G_k$-module, via the action of $\hat{\chi} \colon G_k \longrightarrow \Z_\ell^\times$, the $\ell$-adic cyclotomic character:
\[ \sigma(\mathbf{\upzeta}) := \mathbf{\upzeta}^{\hat{\chi}(\sigma)}, \quad \text{ i.e., } \quad \sigma(\zeta_n) = \zeta_n^{\hat{\chi}(\sigma) \pmod{\ell^n} } \quad (\forall\,n \geq 1). \]
Viewed as a $G_k$-module, $T_\ell \bmu$ is often denoted $\Z_\ell(1)$, and we let $\Z_\ell(-1)$ denote its $\Z_\ell$-dual, i.e. $\Z_\ell(-1) := \Hom_{\Z_\ell}(\Z_\ell(1), \Z_\ell)$.

For any $\Z_\ell[G_k]$-module $M$, and any $n \geq 0$, we define the $n$th Tate twist of $M$ and $(-n)$th Tate twist of $M$, respectively, by
\[ \begin{split}
M(n) & := M \otimes_{\Z_\ell} \Z_\ell(1)^{\otimes n} \\
M(-n) & := M \otimes_{\Z_\ell} \Z_\ell(-1)^{\otimes n}.
\end{split} \]
Suppose $K$ is a field equipped with a ring homomorphism $\Z_\ell \to K$ (e.g., $K = \Q_\ell$ or $K = \F_\ell$). For any $n \in \Z$, we set $K(n) := \Z_\ell(n) \otimes_{\Z_\ell} K$. If $V$ is a $K[G_k]$-module, we define the $n$th Tate twist of $V$ by
\[ V(n) := V \otimes_{K} K(n). \]
If $\dim_{K} V = t$, then the $G_k$-action on $V$ yields a representation
\[ \rho \colon G_k \longrightarrow GL(V) \cong GL_t(K), \]
and we write $\rho(n)$ for the $G_k$-representation attached to the Tate twist $V(n)$.

We return to studying the $G_k$-action on the $\F_\ell$-vector space $A_s[\ell]$. For each $0 \leq i \leq \varphi(\ell^s)$, set $V_i := \varepsilon^i A_s[\ell]$, yielding a chain
\[ \{0 \} = V_{\varphi(\ell^s)} \leq V_{\varphi(\ell^s)-1} \leq \cdots \leq V_1 \leq V_0 = A_s[\ell]. \]
Notice that $V_{\varphi(\ell^s)-1} = W_s$ and so is $G_k$-stable.
\begin{lemma}\label{lemma:Vi_props}
The spaces $V_i$ satisfy the following properties.
\begin{enumerate}[label=(\alph*)]
\item For each $0 \leq i \leq \varphi(\ell^s)$, the subspace $V_i$ is $G_k$-stable.
\item For each $0 \leq i < \varphi(\ell^s)$, $\dim_{\F_\ell} V_i/V_{i+1} = m$.
\item For each $0 \leq i \leq \varphi(\ell^s)-2$, there is an isomorphism of $\F_\ell[G_k]$-modules,
\[ \frac{V_i}{V_{i+1}}(1) \cong \frac{V_{i+1}}{V_{i+2}}. \]
\end{enumerate}
\end{lemma}
Before proving the lemma, we note the following immediate corollary, stemming from the observation that
$W_s = V_{\varphi(\ell^s)-1}$.
\begin{corollary}\label{cor:As_block_diagonal}
Let $\psi_s \colon G_k \to GL_m(\F_\ell)$ be the Galois representation attached to $W_s$. Then the Galois representation attached to $A_s[\ell]$ has the block upper triangular form
\[ \rho_{A_s,\ell} \sim \begin{pmatrix} \psi_s & * & \cdots & * \\ & \psi_s(-1) & \cdots & * \\ & & \ddots & \vdots \\ & & & \psi_s(1-\varphi(\ell^s)) \end{pmatrix}. \]
\end{corollary}
\begin{proof}[Proof of Lemma \ref{lemma:Vi_props}]
Part (b) follows easily from Lemma \ref{lemma:Vepsilon}. Note that for any $\sigma \in G_k$ we have
\[ \sigma(\varepsilon) = \sigma(\zeta_s - 1) = \zeta_s^{\chi(\sigma)} - 1 = (\varepsilon + 1)^{\chi(\sigma)} - 1 = \varepsilon( \chi(\sigma) + \varepsilon h(\varepsilon)) \]
for some polynomial $h$ with coefficients in $\Z_\ell$. Thus, $\sigma(\varepsilon) = u \cdot \varepsilon$ for some unit $u \in \Z_\ell[\zeta_s]^\times$.

Any element of $V_i$ has the form $\varepsilon^i P$, where $P \in A_s[\ell]$. Accordingly,
\[ \sigma(\varepsilon^i P) = u^i \varepsilon^i P^\sigma \in V_i, \]
since $A_s[\ell]$ is $G_k$-stable. We deduce every $V_i$ is $G_k$-stable, verifying (a).

To demonstrate part (c), it suffices to show that the following diagram of $\F_\ell$ vector spaces is commutative for any $\sigma \in G_k$:
\[
\begin{tikzcd}[column sep=large]
V_i/V_{i+1} \ar{rr}{\cdot \varepsilon} \ar{d}[swap]{\sigma} & & V_{i+1}/V_{i+2} \ar{d}{\sigma} \\
V_i/V_{i+1} \ar{r}[swap]{\cdot \varepsilon \vphantom{(\chi)}} & V_{i+1}/V_{i+2} \arrow{r}[swap]{\cdot \chi(\sigma)} & V_{i+1}/V_{i+2}
\end{tikzcd}
\]
Indeed, suppose $\varepsilon^i P$ is an element of $V_i$. Then
\[ \begin{split}
\sigma( \varepsilon \cdot \varepsilon^i P) & = \sigma(\varepsilon) \cdot \sigma( \varepsilon^i P) \\
& = \varepsilon( \chi(\sigma) + \varepsilon h(\varepsilon)) \cdot (u^i \varepsilon^i P^\sigma) \\
& = \chi(\sigma) u^i \varepsilon^{i+1} P^\sigma + u^i \varepsilon^{i+2} h(\varepsilon) P^\sigma \\
& \equiv \chi(\sigma) \cdot \varepsilon \cdot \sigma(\varepsilon^i P) \pmod{V_{i+2}}.
\end{split} \]
So $V_{i+1}/V_{i+2}$ is isomorphic to the twist $(V_i/V_{i+1})(1)$, as claimed.
\end{proof}

\begin{proof}[Proof of Theorem \ref{thm:main_theorem}]
Suppose $C/k$ is a geometrically irreducible superelliptic curve of degree $\ell^n$, and suppose $J = \Jac(C)$ has good reduction away from $\ell$. Let $S$ be the branch locus of the covering $C \to \P^1$, and suppose $k(S) \subseteq \ten(k,\ell)$. We want to demonstrate $\rho_{J,\ell}$ is upper triangular with powers of $\chi$ on the diagonal (up to conjugacy). Combining Proposition \ref{prop:J_As_equiv} and Corollary \ref{cor:As_block_diagonal}, it is enough to show $\psi_s$ is of this form for every $s$. It follows from Lemma \ref{lemma:Ws_Gk_stable} that $k(W_s) \subseteq \ten(k,\ell)$. Thus, again appealing to Lemma 3.3 of \cite{Rasmussen-Tamagawa:2017}, we see every $\psi_s$ has the desired form.
\end{proof}

\section{Comparison of $k(W_s)$ and $k(S[\less s])$}

In this section, we describe the circumstances under which $k(W_s)$ and $k(S[\less s])$ may coincide. We continue the notations of the previous sections.
We have the following chain of $\F_\ell[G_k]$-modules:
\[ 0 \leq \langle \overline{\mathfrak{R}}_s \rangle \leq \mathcal{W}_s^0 \otimes \F_\ell \leq \mathcal{W}_s \otimes \F_\ell. \]
Corresponding to this chain, we have a $G_k$-representation attached to $\mathcal{W}_s \otimes \F_\ell$ which is block upper triangular, comprised of three blocks on the diagonal. By definition of $\mathcal{W}^0_s$, $\mathcal{W}^0_s \otimes \F_\ell$ has codimension $1$ inside $\mathcal{W}_s \otimes \F_\ell$. Thus, the final block is $1$-dimensional; likewise, the first block is $1$-dimensional, corresponding to $\langle \overline{\mathfrak{R}}_s \rangle$. In fact, these one-dimensional blocks are both trivial: $\overline{\mathfrak{R}}_s$ is fixed by $G_k$, and the $G_k$-action on $\mathcal{W}_s \otimes \F_\ell$ preserves degree and thus $(\mathcal{W}_s \otimes \F_\ell) / (\mathcal{W}^0_s \otimes \F_\ell)$ has trivial $G_k$-action. By Lemma \ref{lemma:Ws_FlGk}, the intermediate block is in fact given by $\psi_s$.
\begin{lemma}\label{lemma:Hs_lgroup}
For any $1 \leq s \leq n$, we have
\[ k(W_s) \subseteq k(\mathcal{W}^0_s \otimes \F_\ell) \subseteq k(S[\less s]). \]
Moreover, the group $H_s := \Gal \bigl( k(S[\less s])/k(W_s) \bigr)$ is an $\ell$-group.
\end{lemma}
\begin{proof}
The containments are immediate (for the second, apply Lemma \ref{lemma:WsFl}). Let $\psi$ be the $G_k$-representation attached to $\mathcal{W}_s \otimes \F_\ell$. By the preceding discussion, $\psi$ has the form
\[ \mathcal{\psi} \sim \begin{pmatrix} 1 & * & * \\ & \psi_s & * \\ & & 1 \end{pmatrix}, \]
and $\psi_s$ is itself upper triangular by assumption. Thus, the group $H_s$ is given by $\mathcal{\psi}(\ker \psi_s)$, which is unit upper triangular and so must be an $\ell$-group.
\end{proof}
\begin{proposition}\label{prop:kS_less_S0}
For any $1 \leq s \leq n$, $k(S[\less s] - S[0]) \subseteq k(W_s)$.
\end{proposition}
\begin{proof}
Suppose $\sigma \in G_k$ fixes $k(W_s)$ and $\xi \in S[j]$ with $j \neq 0$. Notice that in this case, $\overline{D}_{s,\xi} = \overline{[\xi]}_s$. Moreover, by the $G_k$-equivariance of $[\,\cdot\,]$,
\[ \overline{D}_{s,\sigma(\xi)} = \overline{[\sigma(\xi)]}_s = \sigma \overline{[\xi]}_s = \sigma \overline{D}_{s,\xi}. \]
Thus, as $\sigma$ fixes $W_s$ pointwise, we have
\[ \begin{split}
(\alpha_s \otimes 1)(\overline{[\sigma(\xi)]}_s) & = (\alpha_s \otimes 1)(\overline{D}_{s,\sigma(\xi)}) \\
& = \Delta_{s,\sigma(\xi)} = \sigma \cdot \Delta_{s,\xi} = \Delta_{s,\xi} \\
& = (\alpha_s \otimes 1)(\overline{D}_{s,\xi}) = (\alpha_s \otimes 1)(\overline{[\xi]}_s).
\end{split} \]
So $\overline{[\sigma(\xi)]}_s - \overline{[\xi]}_s \in \ker (\alpha_s \otimes 1) = \langle \overline{\mathfrak{R}}_s \rangle$. If $[\xi]_s \neq [\sigma(\xi)]_s$ in $\mathcal{W}_s$, then it must be the case that $\overline{[\xi]}_s - \overline{[\sigma(\xi)]}_s$ is a nonzero multiple of $\overline{\mathfrak{R}}_s$. Since the divisors $[\eta]_s$, $[\eta']_s$ are disjoint for any distinct $\eta, \eta' \in S[\less s]$, it follows from Lemma \ref{lemma:Rsbar} that $S[\less s] = \{\xi, \sigma(\xi) \} \subseteq S[j]$. This is absurd, since we know $S[0]$ is nonempty. Thus, $[\xi]_s = [\sigma(\xi)]_s$ for every $\xi \in S[\less s] - S[0]$, and so $\sigma$ fixes $k(S[\less s] - S[0])$ as required.
\end{proof}

Let $d_s := [k(S[\less s]) : k(W_s)] = \#H_s$. In the remainder of this section, we prove the following refinement.
\begin{proposition}\label{prop:exceptional_cases}
Suppose $C/k$ satisfies the hypotheses of Theorem \ref{thm:main_theorem}. For any $1 \leq s \leq n$, the extension $k(S[\less s])/k(W_s)$ is Galois and $d_s \leq 4$. Moreover, one of the following must occur:
\begin{enumerate}[label=(\alph*)]
\item $d_s = 1$, i.e., $k(S[\less s]) = k(W_s)$,
\item $\ell = 3$, $S[\less s] = S[0]$, $r_0 = 3$, $s = 1$, and $d_s = 3$,
\item $\ell = 2$, $S[\less s] = S[0]$, $r_0 = 4$, $s \leq 2$, and $d_s = 4$,
\item $\ell = 2$, $S[\less s] = S[0]$, $r_0 = 4$, and $d_s = 2$,
\item $\ell = 2$, $S[\less s] = S[0]$, $r_0 = 2$, $s = 1$, and $d_s = 2$,
\item $\ell = 2$, $S[\less s] \neq S[0]$, $r_0 = 2$, and $d_s = 2$.
\end{enumerate}
\end{proposition}
We recall some needed constructions. For any set $X$, we let $\triangle X$ denote the image of the diagonal inclusion $X \to X \times X$. There is a natural action of the symmetric group $\mathfrak{S}_2$ on $X \times X - \triangle X$, and we set
\[ {\wedge}^2 X := (X \times X - \triangle X)/\mathfrak{S}_2. \]
If $X$ is a $G$-set for some group $G$, then both $X \times X - \triangle X$ and ${\wedge}^2 X$ inherit the structure of a $G$-set naturally.
\begin{lemma}\label{lemma:Gk_inj}
Let $X$ and $Y$ be $G_k$-sets. Suppose any of the following conditions hold:
\begin{enumerate}[label=(\alph*), leftmargin=*]
\item there exists an injective $G_k$-set map $\Psi_0 \colon X \to Y$,
\item there exists an injective $G_k$-set map $\Psi \colon (X \times X - \triangle X) \to Y$,
\item there exists an injective $G_k$-set map $\Psi' \colon {\wedge}^2 X \to Y$, and ${\#X \neq 2}$.
\end{enumerate}
Then $k(X) \subseteq k(Y)$.
\end{lemma}
\begin{proof}
That condition (a) is sufficient follows from the equivariance of a $G_k$-set map. For condition (b), suppose $\sigma \in G_k$ fixes $Y$ pointwise. By the assumed injectivity of $\Psi$, $\sigma$ fixes every point of $X \times X - \triangle X$. But the obvious map $\Aut(X) \to \Aut(X \times X - \triangle X)$ is injective, and so $\sigma$ fixes $X$ pointwise also. The same argument applies for condition (c), except that the analogous map $\Aut(X) \to \Aut({\wedge}^2 X)$ is injective only if $\#X \neq 2$.
\end{proof}
In light of this, we are encouraged to look for injective set maps into $W_s$ which respect the action of Galois.
\begin{lemma}
The stabilizer in $H_s$ for any $\eta \in S[0]$ is trivial. In particular, if the action of $H_s$ on $S[0]$ has any fixed points (e.g., if $S[0]$ has a $k$-rational point), then $H_s = \{1 \}$.
\end{lemma}
\begin{proof}
Because the original selection of $\xi_0 \in S[0]$ in \S3 was arbitrary, we may assume $\eta = \xi_0$. Now the map
\[ \Psi_0 \colon S[0] - \{\xi_0 \} \longrightarrow W_s, \qquad \xi \mapsto \Delta_{s,\xi}, \]
though not generally $G_k$-equivariant, is $G_{k(\xi_0)}$-equivariant. The independence of $\{D_{s,\xi} : \xi \neq \xi_0 \}$ implies $\Psi_0$ is injective. So
\[ k(S[0]) = k(\xi_0)(S[0] - \{\xi_0 \}) \subseteq k(\xi_0)(W_s), \]
which implies $k(S[\less s]) = k(W_s)(\xi_0)$. Now for any $\eta \in S[0]$,
\[ \Stab_{H_s}(\eta) = \Gal( k(S[\less s])/k(W_s)(\eta) ) = \{1 \}. \qedhere \]
\end{proof}
\begin{corollary}\label{cor:ds_r0}
$d_s \mid r_0$.
\end{corollary}
\begin{proof}
Because all stabilizers are trivial, the orbit of any $\eta \in S[0]$ has size exactly $\#H_s = d_s$.
\end{proof}
\begin{proof}[Proof of Proposition \ref{prop:exceptional_cases}]
If (a) does not hold, then $d_s \neq 1$. We first suppose $\ell \neq 2$. Consider the following $G_k$-set map:
\[ \Psi \colon S[0] \times S[0] - \triangle S[0] \longrightarrow W_s, \qquad (\xi, \eta) \mapsto \Delta_{s,\xi} - \Delta_{s,\eta}. \]
To see that $\Psi$ is in fact $G_k$-equivariant, note $D_{s,\xi} - D_{s,\eta} = [\xi]_s - [\eta]_s$. The $G_k$-equivariance of $\Psi$ now follows from that of $[\,\cdot\,]_s$ and $\alpha_s$. Because we are assuming $H_s$ is nontrivial, $\Psi$ cannot be injective, by Lemma \ref{lemma:Gk_inj}(b) and Proposition \ref{prop:kS_less_S0}. Thus, there exist distinct pairs $(\xi, \eta)$, $(\xi', \eta')$ such that
\[ \Delta_{s,\xi} - \Delta_{s,\eta} - \Delta_{s,\xi'} + \Delta_{s,\eta'} = 0, \qquad \xi \neq \eta, \xi' \neq \eta' \]
in $W_s$. Appealing to the exact sequence established in Lemma \ref{lemma:Ws_FlGk}, this is equivalent to the equality
\[ \overline{D}_{s,\xi} - \overline{D}_{x,\eta} - \overline{D}_{s,\xi'} + \overline{D}_{s,\eta'} = \mu \overline{\mathfrak{R}}_s \]
within $\mathcal{W}^0_s \otimes \F_\ell$. We claim $\mu \neq 0$.

If $\mu = 0$, then lifting back to $\mathcal{W}_s^0$ gives
\[ [\xi]_s - [\eta]_s - [\xi']_s + [\eta']_s \in \ell \mathcal{W}^0_s. \]
Because $[\,\cdot\,]$ produces reduced divisors (and because $\ell \neq 2$), the membership can hold only if this divisor vanishes outright. Further, for any distinct points $\omega, \omega'$, the divisors $[\omega]_s$ and $[\omega']_s$ have disjoint support. Consequently, this vanishing forces $(\xi, \eta) = (\xi', \eta')$, a contradiction. So $\mu \neq 0$.

Recall that $\{ \overline{D}_{s,\xi} : \xi \neq \xi_0 \}$ is a basis for $\mathcal{W}^0_s \otimes \F_\ell$, and that when $\overline{\mathfrak{R}}_s$ is expressed in terms of this basis, every coefficient is nonzero. Thus, we may conclude that $S[\less s] = S[0] = \{\xi, \xi', \eta, \eta'\}$. In particular, $r_0 \leq 4$. However, as $d_s \neq 1$, the action of $H_s$ on $S[0]$ has no fixed points, and so each orbit has cardinality divisible by $\ell$. The only possibility is that $\ell = r_0 = d_s = 3$. At this point, we may assume $C_s$ has an affine model of the form
\[ y^{3^s} = \lambda \prod_{i=0}^2 (x - \xi_i)^{n_i}, \qquad 0 < n_i < 3^n, \quad 3^s \mid (n_0 + n_1 + n_2), \]
where the $H_s$-action on $\{\xi_0, \xi_1, \xi_2 \}$ is transitive. Since $C_s$ is $k$-rational, the exponents $n_i$ are all equal. Consequently, $3^{s-1} \mid n_i$. But already we have $\xi_0 \in S[0]$ and so $s = 1$. In summary, if $d_s \neq 1$ and $\ell \neq 2$, then (b) holds.

For the remainder, we now suppose $d_s \neq 1$ and $\ell = 2$. Consider now the $G_k$-set map
\[ \Psi' \colon {\wedge}^2 S[0] \longrightarrow W_s, \qquad \{\xi, \eta \} \mapsto \Delta_{s,\xi} + \Delta_{s,\eta}. \]
Note that, with $\ell = 2$, we have $\overline{[\xi]}_s + \overline{[\eta]}_s = \overline{D}_{s,\xi} + \overline{D}_{s,\eta}$
in $\mathcal{W}^0_s \otimes \F_\ell$, and so
\[ \Psi' \bigl( \{\xi, \eta \} \bigr) = (\alpha_s \otimes 1) \bigl( \overline{[\xi]}_s + \overline{[\eta]}_s \bigr). \]
By Lemma \ref{lemma:Gk_inj}(c), as $d_s \neq 1$, either $r_0 = 2$ or $\Psi'$ is not injective. Let us first assume $\Psi'$ fails to be injective. Then there exist two distinct subsets $\{\xi, \eta \}$, $\{\xi', \eta' \}$ such that
\[ \overline{[\xi]}_s + \overline{[\eta]}_s + \overline{[\xi']}_s + \overline{[\eta']}_s = \mu \overline{\mathfrak{R}}_s. \]
The divisor on the left can only vanish if $\{\xi, \eta \} = \{\xi', \eta' \}$, and so $\mu \neq 0$. Suppose the two subsets are not disjoint, say with $\xi = \xi'$. Then the relation simplifies to $\overline{[\eta]}_s + \overline{[\eta']}_s = \overline{\mathfrak{R}}_s$. But the support of $\overline{\mathfrak{R}}_s$ is all of $S[\less s]$, and so
$\xi \in S[0] = S[\less s] = \{\eta, \eta' \}$. This violates the assumption that $\{\xi, \eta \}, \{\xi', \eta' \} \in {\wedge}^2 S[0]$. Thus the two sets are disjoint. Hence $S[0] = S[\less s]$ and $r_0 = 4$. We may now assume $C_s$ has the form
\[ y^{2^s} = \lambda \prod_{i=0}^3 (x-\xi_i)^{n_i}, \qquad 0 < n_i < 2^n, \quad 2^s \mid (n_0 + n_1 + n_2 + n_3). \]
By Corollary \ref{cor:ds_r0}, $d_s \in \{2, 4 \}$. If $d_s = 2$, then (d) holds. If $d_s = 4$, then $H_s$ is transitive on $S[0]$, and the $n_i$ are all equal. In this case we have $2^s \mid 4n_0$. But already $\xi_0 \in S[0]$, and so $s \leq 2$. Thus (c) holds.

Finally, let us now take $r_0 = 2$. If $S[0] \neq S[\less s]$, then (f) holds. If $S[0] = S[\less s]$, then we may assume $C_s$ has the form
\[ y^{2^s} = \lambda (x - \xi_0)^{n_0} (x - \xi_1)^{n_1}, \qquad 0 < n_i < 2^n, \quad 2^s \mid (n_0 + n_1). \]
The transitivity of $H_s$ on $S[0]$ implies $n_0 = n_1$, and so $2^s \mid 2n_0$. But $\xi_0 \in S[0]$ and so $s = 1$. In this case, (e) holds.
\end{proof}

\section{An Application}

Recall that a Picard curve is a genus $3$ superelliptic curve $C/k$ and so, possibly after base change, provides a cyclic cover of $\P^1$ of degree $3$. In \cite{Malmskog-Rasmussen:2016}, Malmskog and the first author determined the full set of Picard curves defined over $\Q$ which have good reduction away from $3$. Up to $\Q$-isomorphism, there are $45$ such curves. For a subset of $9$ of these curves, the authors demonstrated that $\Q(\Jac(C)[3^\infty]) \subseteq \ten(\Q, 3)$ by noting an explicit model for the curve and using the criterion of Anderson and Ihara. However, with Theorem \ref{thm:main_theorem}, we may now extend the result to all such curves.
\begin{proposition}
Let $C/\Q$ be a Picard curve with good reduction away from $3$. Then $\Q(\Jac(C)[3^\infty]) \subseteq \ten(\Q,3)$.
\end{proposition}
\begin{proof}
By the results of \cite{Malmskog-Rasmussen:2016}, every such curve has a branch set $S$ satisfying $\Q(S) \subseteq \ten(\Q,3)$. Thus, Theorem B applies.
\end{proof}

\bibliographystyle{halpha}

\def\cprime{$'$}

\end{CJK}
\end{document}